\documentclass[11pt]{amsart}
\usepackage{amsmath,amssymb,mathtools}

\usepackage[english]{babel}

\usepackage{hyperref}

\newtheorem{teo}{Theorem}

\newtheorem{Lemma}[teo]{Lemma}
\newtheorem{Cor}[teo]{Corollary}

 \DeclareMathOperator{\spann}{span}

\newcommand{\er}{\mathbb{R}}
\newcommand{\en}{\mathbb{N}}
 
\newcommand{\zet}{\mathbb{Z}} 
\newcommand{\qu}{\mathbb{Q}}
\newcommand{\ce}{\mathbb{C}}

\newcommand{\norm}[1]{\left\Vert#1\right\Vert}
\newcommand{\abs}[1]{\left\vert#1\right\vert}
\date{\today}
\author{Alessandro Vignati}
\title[Nonseparable amenable]{An algebra whose subalgebras are characterized by density}
\address[Alessandro Vignati]{York University, Department of Mathematics and Statistics, 4700 Keele St., Toronto, Ontario, Canada M3J 1P3}

\email{ale.vignati@gmail.com}
\subjclass[2000]{	46L05, 47L40}%
\keywords{amenable Banach algebra, Luzin gap, nonseparable algebras, unitarizable group representation}
\begin{document}

\begin{abstract}
We refine a construction of Choi, Farah and Ozawa to build a nonseparable amenable operator algebra $\mathcal A\subseteq\ell_\infty(M_2)$  whose nonseparable subalgebras, including $\mathcal A$, are not isomorphic to a $C^*$-algebra. This is done using a Luzin gap and a uniformly bounded group representation.

Next, we study additional properties of $\mathcal A$ and of its separable subalgebras, related to the Kadison Kastler metric.
\end{abstract}
\maketitle
\section{Introduction}
We focus on the problem of whether an amenable operator algebra (i.e., a Banach subalgebra of $\mathcal B(H)$) is necessarily isomorphic to a $C^*$-algebra. Recently this longstanding problem was solved in \cite{CFO13} by giving a nonseparable counterexample. The question of whether a separable counterexample can be provided is still open, although partial results were obtained: it was shown in \cite{gifford2006} that a separable counterexample cannot be a subalgebra of the compact operators. It has also been proved (see \cite{marcouxpopov2013}) that every abelian amenable operator algebra is similar to a $C^*$-algebra.

We construct a nonseparable amenable operator algebra $\mathcal A$ with the property that none of its nonseparable subalgebras is isomorphic to a $C^*$-algebra, and yet $\mathcal A$ is  an inductive limit of separable subalgebras each of which is isomorphic to a $C^*$-algebra. This is mainly motivated by the construction in \cite{CFO13}, where it is not necessary to use the full power of the set theoretical tool involved, a particular family of subsets of $\en$ known as a Luzin gap.

The main technique we use consists of taking a uniformly bounded representation $\pi$ of an uncountable abelian group $G$ in a corona algebra $\mathcal C$ with quotient map $Q$, and by considering the algebra $\mathcal A=Q^{-1}(\overline\spann\pi(G))$ as our example. The definition of uniformly bounded representation will be given in section \ref{section1}. This representation has an even more striking (although easier to prove) property: for a subgroup $H$ of $G$, the restriction of $\pi$ to $H$ is unitarizable if and only if $H$ is countable. In terms of the early version of \cite{CFO13} (see \cite{FO13}), the first bounded cohomology group $H^1_b(H,\mathcal C)$ is trivial if and only if $H$ is countable. Similar phenomena occurring at the least uncountable cardinal $\aleph_1$, as well as their connection to cohomology, are well-known (see \cite{Talayco95}). 

The next part of the paper is motivated by a question of Luis Santiago. In Theorem \ref{thethesis2}, we show that there exists a $C^*$-algebra $C$ such that for every $\epsilon>0$ there is an amenable operator algebra $\mathcal A_\epsilon$ that is $\epsilon$-close to $C$ in the Kadison-Kastler metric (see \cite{KadisonKastler72} for basic definitions) and that is not isomorphic to a $C^*$-algebra. All considered algebras are separably representable.

We should notice that these techniques cannot be used to provide a separable counterexample. In fact, whenever $G$ is an amenable group and $\pi$ is a uniformly bounded representation in a corona of a $\sigma$-unital algebra with quotient map $Q$, $Q^{-1}(\overline\spann\pi(G))$ is an amenable operator algebra. Although, thanks to a model theoretical property carried by coronas of $\sigma$-unital algebras and known as countable degree-1 saturation  (see \cite{farah2011countable}), if $G$ is countable the algebra $Q^{-1}(\overline\spann\pi(G))$ is automatically isomorphic to a $C^*$-algebra (\cite[Theorem 8]{CFO13}) and, applying the solution to the Kadison similarity problem for nuclear $C^*$-algebras (see \cite{pisier01}), even similar to a $C^*$-algebra.

The author would like to thank Ilijas Farah, Luis Santiago, Stuart White and the anonymous referees for the impressive number of suggestions received.
\section{The main construction}\label{section1}
The main Theorem that we are going to prove is the following.

\begin{teo}\label{thethesis}There is a nonseparable amenable operator algebra $\mathcal A$ such that there is no nonseparable algebra $\mathcal B\subseteq\mathcal A$ that is isomorphic to a $C^*$-algebra.
\end{teo}
 The Kadison-Kastler distance of two subalgebras $A,B\subseteq\mathcal B(H)$ is defined as \[d_K(A,B)=\max\{\sup_{x\in A_1}\inf_{y\in B_1}\norm{x-y},\sup_{x\in B_1}\inf_{y\in A_1}\norm{x-y}\}\] where $A_1$ and $B_1$ are the sets of elements of norm $1$ in $A$ and $B$ respectively.

Looking at the properties of $\mathcal A$ as in Theorem \ref{thethesis} we can formulate the following:
 \begin{teo}\label{thethesis2}There is a $C^*$-algebra $C$ such that for any $\epsilon>0$ there is a nonseparable amenable algebra $\mathcal A_\epsilon$ such that no nonseparable subalgebras of $\mathcal A_\epsilon$ is isomorphic to a $C^*$-algebra and  $d_K(\mathcal A_\epsilon,C)<\epsilon$.
\end{teo}

Let $\ell_\infty(\en,M_2)$ be the unital $C^{*}$-algebra of bounded sequences in $M_2(\ce)$ and let \[c_0(\en,M_2)=\{(x_n)\in\ell_\infty(\en,M_2)\mid \lim_n\norm{x_n}=0\}.\] We have that $c_0(\en,M_2)$ is a two-sided closed ideal in $\ell_\infty(\en,M_2)$, hence it is automatically self-adjoint. Let $\mathcal C(\en,M_2)$ be the quotient and $Q$ be the quotient map. We will write $a\sim_\epsilon b$ for $\norm{a-b}<\epsilon$ when $a$ and $b$ sit in the same normed algebra. We should point out that this is not an equivalence relation.

If $A$ is a unital $C^*$-algebra, a function $\pi\colon G\to A$ is a \emph{uniformly bounded representation} if $\pi(gh)=\pi(g)\pi(h)$, $\pi(g)$ is invertible for all $g,h\in G$ and $\norm{\pi}=\sup_g\norm{\pi(g)}<\infty$.

For $x,\alpha,\beta\in\ce$ let \[M_{\alpha,\beta,x}=\alpha\left(\begin{array}{cc}
1 & 0 \\
x & -1 \\\end{array}\right)+\beta\cdot I.\]
\begin{Lemma}
\label{compactness}
Fix $s\neq t\in[0,1]$ and $K_1,K_2\in\er^+$. Then there is $C=C(s,t,K_1,K_2)>0$ such that \[d(uM_{\alpha,\beta,s}u^{-1},\mathcal U)+d(uM_{\gamma,\delta,t}u^{-1},\mathcal U)>C\] whenever $2\geq \norm{\alpha},\norm{\gamma}\geq K_2$, $\norm{u},\norm{u^{-1}}\leq K_1$ and $\beta,\delta\in\ce$.
\end{Lemma}
\begin{proof}
Suppose $x,y\in[0,1]$, with $x\neq y$. Let $\alpha,\gamma\neq 0$ and $\beta,\delta\in\ce$. Then it is impossible to simultaneously unitarize $M_{\alpha,\beta,x}$ and $M_{\gamma,\delta,y}$.

To see this let $u$ be an invertible matrix that unitarizes $M_{\alpha,\beta,x}$.
We may assume that $u$ is positive, by the polar decomposition, hence so are $u^{-1}$ and $u^{-2}$ as well. Let $u^{-2}=\left(\begin{array}{cc}
a & b \\
\overline b & c \\\end{array}\right)$. Note that, since the determinant of a positive invertible matrix is positive, we have $a,c\neq 0$. By positivity we have that \[(uM_{\alpha,\beta,x}u^{-1})^*=(u^{-1})^*M_{\alpha,\beta,x}^*u^*=u^{-1}M_{\alpha,\beta,x}^*u,\] hence we have that \[uM_{\alpha,\beta,x}u^{-2}M_{\alpha,\beta,x}^*u=I\] and in particular $M_{\alpha,\beta,x}u^{-2}M_{\alpha,\beta,x}^*=u^{-2}$, that means \[\left(\begin{array}{cc}
\alpha+\beta & 0 \\
\alpha x & \beta-\alpha \\\end{array}\right)\left(\begin{array}{cc}
a & b \\
\overline b & c \\\end{array}\right)\left(\begin{array}{cc}
\overline{\alpha+\beta} & \overline\alpha x \\
0 & \overline{\beta-\alpha} \\\end{array}\right)=\left(\begin{array}{cc}
a & b \\
\overline b & c \\\end{array}\right).\]
Doing the calculation we have \[\left(\begin{array}{cc}
a(\alpha+\beta) & b(\alpha+\beta) \\
a\alpha x+\overline b(\beta-\alpha) & b\alpha x+c(\beta-\alpha) \\\end{array}\right)\left(\begin{array}{cc}
\overline{\alpha+\beta} & \overline\alpha x \\
0 & \overline{\beta-\alpha} \\\end{array}\right)=\left(\begin{array}{cc}
a & b \\
\overline b & c \\\end{array}\right)\] and looking at the first row multiplied by the second column of the latter we have
\[a(\alpha+\beta)\overline\alpha x+b(\alpha+\beta)(\overline{\beta-\alpha})=b\] that is, multiplying both left and right side for $\overline{(\alpha+\beta)}$,  \[a\overline\alpha x=b(\overline\beta+\overline\alpha)-b(\overline\beta-\overline\alpha)=2b\overline\alpha.\] Since $\alpha\neq 0$, we have $ax=2b$, that means that $x$ is unique, once the unitary is fixed.

The thesis is obtained since $\alpha$ and $\gamma$ are quantified over the compact space $\{(x,y)\in\ce^2\mid K_2\leq\abs{x},\abs{y}\leq 2\}$.
\end{proof}
We will write $M_x$ for $M_{1,0,x}$ and $C(x,y,K_1)$ for $C(x,y,K_1,1)$.

To obtain the thesis of Theorem \ref{thethesis} we are also going to use the  full power obtained from the construction of a Luzin gap (see \cite{Luz47}). A proof of the existence of such an object can be found in \cite[Appendix B]{CFO13}.
\begin{Lemma}\label{luzin}
There is a family $\{A_\alpha\mid\alpha\in\aleph_1\}$ of infinite subsets of natural numbers such that 
\begin{enumerate}
\item\label{cond1} if $\alpha\neq\beta$ then $A_\alpha\cap A_\beta$ is finite.
\item\label{cond2} for any $\alpha\in\aleph_1$ and $m\in\en$ the set $\{\beta<\alpha\mid A_\alpha\cap A_\beta\subseteq m\}$ is finite.
\end{enumerate}
\end{Lemma}
A family $\{A_\alpha\}$ with properties \ref{cond1} and \ref{cond2} as in Lemma \ref{luzin} is a \emph{Luzin gap}.

Let $\{A_\alpha\}$ be a Luzin gap, $f\colon\mathcal P(\en)\to[0,1]$ be the canonical bijection onto the Cantor set and $x_\alpha=f(A_\alpha)$.
Let $s_\alpha=\left(\begin{array}{cc}
1 & 0 \\
x_\alpha & -1 \\\end{array}\right)\in M_2$ and $w_\alpha\in \ell_\infty(\en,M_2)$ be defined as $(w_\alpha)_n=s_\alpha$ if $n\in A_\alpha$ and $(w_\alpha)_n=I_{M_2}$ otherwise. We will write $\bar w_\alpha=Q(w_\alpha)$.
Let $G=\bigoplus_{\alpha\in\aleph_1}\zet/2\zet$. Note that we may identify $G$ with $[\aleph_1]^{<\aleph_0}$, the set of all finite subsets of $\aleph_1$, with the operation of symmetric difference (i.e., $ab=(a\cup b)\setminus (a\cap b)$). From now on we will talk about the elements of $G$ as finite subsets of $\aleph_1$, and we will consider $\{\{\alpha\}\}_{\alpha\in\aleph_1}$ as the standard basis for $G$.  Let $\pi\colon G\to\mathcal C(\en,M_2)$ be defined as $\pi(\{\alpha\})=\bar w_\alpha$ and $\pi(\emptyset)=I_{ C(\en,M_2)}$. We have that $\pi$ can be extended to a uniformly bounded representation setting $\pi(s)=\prod_{\alpha\in s}\pi(\{\alpha\})$, since for every $s\in[\aleph_1]^{<\aleph_0}$ we have that \[|\bigcup_{\alpha\neq\beta\in s}(A_\alpha\cap A_\beta)|<\aleph_0,\] and therefore $\norm{\pi(s)}\leq 2$.

Note that \[\pi(\{\alpha_1\})\pi(\{\alpha_2\})=\pi(\{\alpha_1\})+\pi(\{\alpha_2\})-I_{ C(\en,M_2)},\] and for the same reason, for $s=\{\alpha_1,\ldots,\alpha_N\}\in G$ we have \[\pi(s)=-(N-1)\cdot I_{ C(\en,M_2)}+\sum_{j=1}^N\pi(\{\alpha_j\}).\] 

Let $\mathcal A=Q^{-1}(\overline\spann\pi (G))$. We will show that $\mathcal A$ satisfies the conclusion of Theorem \ref{thethesis}. The structure of this algebra depends only on the $A_\alpha$'s and on $\overline x=\{x_\alpha\}_{\alpha\in\aleph_1}\subseteq[0,1]$ hence, fixing the Luzin gap once for all, we will refer to this algebra as $\mathcal A_{\overline x}$.

We will deal with two cases separately. The first case, that is proven in Lemma \ref{subgroups}, occurs when $\mathcal B$ is of the form $\mathcal B=Q^{-1}(\overline\spann \pi(H))$ for some uncountable subgroup $H\subseteq G$ and the second one treats subalgebras of $\mathcal A$ that are not of that form. Note that the proof for the second case also takes care of the first situation.

Recall that for a group $G$ and a Banach space $X$ the group $H_b^1(G,X)$ is the first bounded cohomology group, defined as the linear space of cocycle modulo inner cocycles (see \cite{monod2001continuous} for definitions and properties).
\begin{Lemma}\label{subgroups}
Let $G$ and $\pi$ be as before.
 For subgroup $H\subseteq G$ the following conditions are equivalent:
\begin{enumerate}
\item\label{cond11} $Q^{-1}(\overline\spann\pi(H))$ is isomorphic to a $C^*$-algebra
\item\label{cond12} $\pi\restriction H$ is unitarizable
\item\label{cond13} $H$ is countable.
\item\label{cond14} $H_b^1(H,\mathcal C(\en,M_2))=0$.
\end{enumerate}
\end{Lemma}
\begin{proof}
\ref{cond11}$\iff$ \ref{cond12} is \cite[Lemma 2]{CFO13}, while \ref{cond13} implies \ref{cond12} is proved in \cite[Theorem 8]{CFO13}, and the equivalence of \ref{cond12} and \ref{cond14} is proved in \cite[Section 3]{FO13}. Assume that \ref{cond12} implies \ref{cond13} is false and fix $H$ uncountable subgroup of $G$ with $u\in\mathcal C(\en, M_2)$ that unitarizes $\pi\restriction H$.

We first analyze the special case where there are uncountably many $\alpha$ such that $\{\alpha\}\in H$. Denote $X_H=\{x_\alpha\mid \{\alpha\}\in H\}$ and take $y_1\neq y_2$ two complete accumulation points of $X_H$. We recall that for $X\subseteq [0,1]$ a complete accumulation point for $X$ is a point $x\in[0,1]$ such that $\forall\epsilon>0$ we have \[|X\cap(x-\epsilon,x+\epsilon)|=|X|.\]
Since the representation is unitarizable, for any sequence $\{u_n\}$ of invertibles that represents $u$ we have that, if $\{\alpha\}\in H$ then \[\lim_{n\in A_\alpha}d(u_ns_{\alpha}u_n^{-1},\mathcal U)=0\]
Replacing $u$ with $(uu^*)^{1/2}$, we can assume that $u$ is positive and, since $u$ is invertible we can consider $K>0$ such that $\frac{1}{K}\leq u\leq K$. Fix a sequence $u_n$ of positive and invertible elements such that $u_n$ represents $u$ and $\norm{u_n},\norm{u_n^{-1}}\leq K^2$ for all $n\in\en$.

Let $\epsilon>0$ be such that \[\epsilon<\min\{C(y_1,y_2,K^2),\abs{y_1-y_2}/2\}/4.\] Recall that $u$ unitarizes $\pi$, hence \[\lim_{n\in A_{\alpha}}d(u_ns_{\alpha}u_n^{-1},\mathcal U)=0.\]
Consider, for $k=1,2$, \[X_{k}=\{\alpha\in X_H\mid \abs{y_k-x_\alpha}<\epsilon/2\}\] and let $n_\alpha$ such that for all $n\geq n_\alpha$, if $n\in A_\alpha$, then \[d(u_ns_{\alpha}u_n^{-1},\mathcal U)<\epsilon/2.\] By a pigeonhole principle we can find, for $k=1,2$, $m_k\in\en$ and $Y_k\subseteq X_k$ with $Y_k$ uncountable and such that $\alpha\in Y_k\Rightarrow n_\alpha=m_k$. Let $N=\max\{m_1,m_2\}$ and take a countable subset $F\subseteq Y_1$ and $\gamma>\sup F$ such that $\gamma\in Y_2$. By condition \ref{cond2} of Lemma \ref{luzin} there are $m>N$ and $\alpha\in F$ such that $m\in A_\alpha\cap A_\gamma$, but then \[d(u_ms_{y_1}u_m^{-1},\mathcal U)\sim_{\epsilon/2} d(u_ms_{\alpha}u_m^{-1},\mathcal U)<\epsilon/2\] and \[d(u_ms_{y_2}u_m^{-1},\mathcal U)\sim_{\epsilon/2} d(u_ms_{\gamma}u_m^{-1},\mathcal U)<\epsilon/2,\]  contradicting the definition of $C(y_1,y_2,K^2)$.

We now consider the general case, where there are not necessarily uncountably many singletons in $H$. By the $\Delta$-system Lemma (see \cite[Theorem 2.1.6]{kunen:settheory}), for every uncountable $B\subseteq[\aleph_1]^{<\aleph_0}$, there is $B_1\subseteq B$ uncountable and $r\in [\aleph_1]^{<\aleph_0}$ such that $x\neq y\in B_1$ implies $x\cap y=r$. Such a subfamily is usually called a $\Delta$-system. Since $H$ is closed by symmetrical difference, we can find $n>1$ and $\{Z_\alpha\}_{\alpha\in\aleph_1}\subseteq H$ such that for $\alpha\neq\beta$ we have $Z_\alpha\cap Z_\beta=\emptyset$ and $|Z_\alpha|=n$.
For any $\alpha\in\aleph_1$ let $n_\alpha$ such that \[n_\alpha\geq \max\{A_i\cap A_j\mid i,j\in Z_\alpha\}.\] By  a cardinality argument we can say that there is $\overline n\in\en$ and an uncountable $Y\subseteq\aleph_1$ such that \[\alpha,\beta\in Y\Rightarrow \overline n= n_\alpha=n_\beta.\] Suppose now that $u$ unitarizes $\pi\restriction H$ and take a sequence $u_n$ representing $u$ as above, where each $u_n$ is invertible and positive. We have that \[\pi(Z_\alpha)=\prod_{i\in Z_\alpha}\pi(\{i\}).\] On the other hand we have that $\{A_i\}_{i\in Z_\alpha}$ are disjoint above $\overline n$ hence, for $n\geq\overline n$, we can repeat the argument from the first case, by uncountability of $Y$.
\end{proof}

\begin{proof}[Proof of Theorem \ref{thethesis}]
We need to prove Theorem \ref{thethesis} for a general (i.e., not of the form $Q^{-1}(\overline\spann\pi(H))$ for some uncountable subgroup H), nonseparable subalgebra of $\mathcal A$.

For the sake of obtaining a contradiction, let $\mathcal B\subseteq\mathcal A$ be a nonseparable unital subalgebra and suppose that $\mathcal B$ is isomorphic to a $C^*$-algebra. Let $\Phi\colon\mathcal B\to D$ be a Banach algebra isomorphism, where $D$ is a $C^*$-algebra. Since $\mathcal B$ is unital, so is $D$, and, being a $C^*$-algebra, is generated by its unitaries.  Consider $\Phi^{-1}(\mathcal U(D))$: this is a uniformly bounded subgroup of invertible elements of $\mathcal B$ and by the main result of \cite{Vasilescu} it is similar to a group of unitaries, via a $u\in\ell^\infty(M_2)$.
Therefore $u\mathcal Bu^{-1}$ is a $C^*$-algebra. Note that $u\mathcal Bu^{-1}$ is not necessarily equal to $D$, but it is isomorphic to it via $x\mapsto\Phi(u^{-1}xu)$.
Consider now the set $\{a\in\mathcal B\mid uau^{-1}\text{ is unitary}\}$. This set is nonseparable and, since the density character of $\mathcal B$ is $\aleph_1$, we can extract an uncountable set $\{a_i\}_{i\in\aleph_1}$ such that there is $\epsilon>0$ such that for all $i,j\in\aleph_1$ we have $Q(a_i)\nsim_\epsilon Q(a_j)$.

We have that, for all $i$, $Q(a_i)\in\overline{\spann}\{\{\pi(\{\alpha\}_{\alpha\in\aleph_1}\},1\}$, hence, in particular, for all $i$ there are increasing $\{\alpha_{i,k}\}_{1\leq k}\subseteq\aleph_1$ and $\{c_{i,k}\}_{0\leq k}\subseteq\ce$ such that \[Q(a_i)=c_{i,0}+\sum_{1\leq k}c_{i,k}\pi(\{\alpha_{i,k}\}).\]

Let $\epsilon>0$ be such that for $i,j\in\aleph_1$ we have $Q(a_i)\nsim_{8\epsilon} Q(a_j)$. We can find, for all $i\in\aleph_1$, a minimum $n=n(i,\epsilon)$ and $d_{i,k}\in\qu+\sqrt{-1}\qu$ such that \[Q(a_i)\sim_\epsilon d_{i,0}+\sum_{1\leq k\leq n}d_{i,k}\pi(\{\alpha_{i,k}\}).\] Note that we can assume that for all $j$ we have $n(i,\epsilon)=n(j,\epsilon)$. By countability of $\qu$ we can go to an uncountable subset, re-index it and obtain that for all $k\leq n$ and all $i,j\in\aleph_1$ we have $d_{i,k}=d_{j,k}$. Note that $c_{i,k}\sim_\epsilon d_{i,k}$. Apply the $\Delta$-system lemma in order to have an uncountable $B_1$ such that $\{\alpha_{i,0},\ldots,\alpha_{i,n}\}_{i\in B_1}$ forms a $\Delta$-system. From this, the fact that for all $i\in\aleph_1$ and $k\in\en$ we have $\alpha_{i,k}<\alpha_{k+1}$, together with $Q(a_i)\nsim_{8\epsilon} Q(a_j)$, implies that \[\exists \overline k (|d_{i,\overline k}|>2\epsilon\wedge\forall i,j\in B_1 (\alpha_{i,\overline k}\neq\alpha_{j,\overline k})).\]

 Take $y_1,y_2$ two complete accumulation points of $\{x_{\alpha_{i,\overline k}}\}_{i\in B_1}$ and $C=C(y_1,y_2,K_1,\epsilon/2)$ according to Lemma \ref{compactness}, where $K_1>\norm{u}+\norm{u^{-1}}$ and let $\delta=\min\{\frac{\abs{y_1-y_2}}{8},\frac{C}{4K_1}\}$.
For each $i\in B_1$ there is $N=n(i,\delta)$ and $e_{i,0},\ldots,e_{i,N}$ such that
\[Q(a_i)\sim_\delta e_{i,0}+\sum_{1\leq k\leq N}e_{i,k}\pi(\{\alpha_{i,k}\}).\] By minimality of the choice of $n(i,\delta)$ and $n(i,\epsilon)$ we have that \[n(i,\delta)\geq n(i,\epsilon)\geq\overline k.\] Note that we have $e_{i,k}\sim_\delta c_{i,k}\sim_\epsilon d_{i,k}$ since $\delta<\epsilon/2$, hence, since $\abs{d_{i,\overline k}}>2\epsilon$, we get that $\abs{e_{i,\overline k}}>\epsilon/2$ for all $i\in B_1$.

Let $b_i\in \mathcal \ell_\infty(\en,M_2)$ be defined as
 \[
(b_i)_m=\begin{cases}
e_{i,k}\left(\begin{array}{cc}
1 & 0 \\
x_{\alpha_{i,k}} & -1 \\\end{array}\right)+\displaystyle{\sum_{0\leq l\leq N, \, l\neq k}e_{i,l}\cdot I}&\text{ if }\exists! k\leq N (m\in A_{\alpha_{i,k}})\\
\displaystyle{\sum_{k\leq N}e_{i,k}\cdot I}&\text{ otherwise.}
\end{cases}.\]
Then 
\[Q(b_i)=e_{i,0}+\sum_{1\leq k\leq N}e_{i,k}\pi(\{\alpha_{i,k}\})\sim_\delta Q(a_i).\] 
Consider 
\[X_l=\{\alpha_{i,\overline k}\mid x_{\alpha_{i,\overline k}}\in (y_l+\delta/4,y_l-\delta/4)\},\,\,\, l=1,2.\] 
Both $X_1$ and $X_2$ are uncountable. For every $i$ such that $\alpha_{i,\overline k}\in X_l$ there is $M(i)$ such that $\forall M\geq M(i)$ we have $(a_i)_M\sim_\delta (b_i)_M$ and we can find, for $l=1,2$, $M_l\in\en$ and $Y_l\subseteq X_l$ uncountable such that \[\alpha_{i,\overline k},\alpha_{j,\overline k}\in Y_l\Rightarrow M(i)=M(j)=M_l.\] Take $F\subseteq \{i\mid \alpha_{i,\overline k}\in Y_1\}$ infinite and countable and $i>\sup F$ such that $\alpha_{i,\overline k}\in Y_2$. Then we have that there is $j\in F$ and an index $m>\max (M_1,M_2)$ such that \[m\in A_{\alpha_{i,\overline k}}\cap A_{\alpha_{j,\overline k}}\] by condition 2 of Lemma \ref{luzin}.

Recall that $\delta$ was chosen to be $\delta=\min\{\frac{\abs{y_1-y_2}}{8},\frac{C}{4K_1}\}$ and the latter condition implies that \[0=d(u_{m}(a_i)_mu^{-1}_m,\mathcal U)\sim_\delta d(u_m(b_i)_mu_m^{-1},\mathcal U)\sim_\delta d(u_mM_{e_{i,\overline k},\beta_1,y_2}u_m^{-1})\] for some $\beta_1$ and equivalently for $j$, $M_{e_{j,\overline k},\beta_2,y_1}$ and some $\beta_2$, contradicting the choice of $\delta$ in terms of $C=C(y_1,y_2,K_1,\epsilon/2)$ from Lemma \ref{compactness}.
\end{proof}

We will now focus on the proof of Theorem \ref{thethesis2}.
\begin{Lemma}\label{teo2}
Let $A_\alpha$ be a Luzin gap, $\{x_\alpha\}_{\alpha\in\aleph_1},\{y_\alpha\}_{\alpha\in\aleph_1}\subseteq[0,1]$ and $\mathcal A_{\overline x}$ and $\mathcal A_ {\overline y}$ be constructed as before as inverse images of an uniformly bounded representation of $\bigoplus_{\alpha\in\aleph_1}\zet/2\zet$ inside $\mathcal C(\en,M_2)$. If for all $\alpha$ we have $\abs{x_\alpha-y_{\alpha}}<\epsilon$ then \[d_K(\mathcal A_{\overline x},\mathcal A_{\overline y})<4\epsilon.\]
\end{Lemma}
In particular, if $x_\alpha=z$ for all $\alpha\in\aleph_1$ and a fixed $z\in[0,1]$, the algebra $\mathcal A_{\overline x}$ is isomorphic to a $C^*$-algebra (if $z=0$ it is a $C^*$-algebra itself), since it is always possible to unitarize $\left(\begin{array}{cc}
1 & 0 \\
x_\alpha & -1 \\\end{array}\right)$. The thesis of Theorem \ref{thethesis2} follows as a consequence of the existence of a complete accumulation point for any uncountable subset of $[0,1]$.
\begin{proof}
Let \[s^0_{\alpha}=\left(\begin{array}{cc}
1 & 0 \\
x_\alpha & -1 \\\end{array}\right)\text{ and }s_\alpha^1=\left(\begin{array}{cc}
1 & 0 \\
y_\alpha & -1 \\\end{array}\right)\] and let $p_\alpha^l$ be defined as \[(p_\alpha^l)=\begin{cases}s_\alpha^l & \text{ if } n\in A_\alpha\\
I&\text{otherwise.}
\end{cases}\] for $l=0,1$. By hypothesis $s_\alpha^0\sim_\epsilon s_\alpha^1$ for all $\alpha$.

Let $a\in (A_{\overline x})_1$. Up to $\epsilon$ we can assume that $a$ has finite support. This means that we may assume that there are $n\in\en$,  $\alpha_1,\ldots,\alpha_n$ and $c_{k}$ for $0\leq k\leq n$ such that
\[Q(a)\sim_{\epsilon/2}c_{0}I_{\mathcal C}+\sum_{1\leq k\leq n}c_kQ(p_{\alpha_k}^0)\] with $\abs{c_k}\leq 1$. 

We know that there is $\overline n$ such that for all $m\geq\overline n$ we have $(a)_m\sim_{\epsilon} c_0\cdot I+\sum_{1\leq k\leq n}c_kp_{\alpha_k}^0$.

Let $b$ be defined as $(b)_m=(a)_m$ for \[m\leq\max\{\overline n,\max\bigcup_{1\leq i<j\leq n}A_{\alpha_i}\cap A_{\alpha_j}\}\] and $(b)_m= c_0\cdot I+\sum_{1\leq k\leq n}c_kp_{\alpha_k}^1$ otherwise. Then $b\in\mathcal A_{\overline y}$, and since \[\abs{c_k}\leq 1\text{ and }s_{\alpha_k}^0\sim_{\epsilon}s_{\alpha_k}^1\] we have $a\sim_\epsilon b$. Moreover we have that $\norm{b}\sim_\epsilon 1$, since the norm in $\ell_\infty$ is the sup norm, so $a\sim_{2\epsilon}\frac{b}{\norm{b}}$.
\end{proof}
As a concluding remark we should point out some consequences on the structure of the Kadison-Kastler metric in the set of Banach subalgebras of $\mathcal B(H)$, where $H$ is separable. In order to extend some of the results in \cite{choi1983completely} and as a consequence of Theorem \ref{thethesis2} we have that neither of the sets \[C_*=\{A\subseteq\mathcal B(H)\mid A\text{ is a $C^*$-algebra}\}\] and \[C_{\sim}=\{A\subseteq \mathcal B(H)\mid A\text{ is isomorphic to a $C^*$-algebra}\}\] are open in Kadison-Kastler metric. 
We should point out, thanking Stuart White for the observation, that the fact that $C_*$ is not open follows easily considering $M_2(\ce)$, \[A=\{\left(\begin{array}{cc}
\alpha & 0 \\
0 & \beta \\\end{array}\right)\mid\alpha,\beta\in\ce\},\,\,\,v=\left(\begin{array}{cc}
1 & 0 \\
\epsilon i & 1 \\\end{array}\right)\] and $B=vAv^{-1}$. Then $A$ and $B$ are $2\epsilon$-close in KK-metric, yet $B$ is not a $C^*$-algebra, since $v\in B$ but $v^*\notin B$.

The existence of a separable amenable operator algebra that is not isomorphic to a $C^*$-algebra is however still open. This means that it is not known whether $C_{\sim}$ is open in the subspace topology when intersected with the set of all separable amenable algebras. The fact that the set of separable operator algebras is clopen in the set of all operator algebras (see \cite[Prop 2.10]{perturb} for the nontrivial direction), suggests that having information on what is happening in the nonseparable case will not help to describe the situation in the separable one.

We note also that, for a fixed $\{x_\alpha\}\subseteq (x-\epsilon,x+\epsilon)$, any permutation  of $\aleph_1$ induces the construction of a different, non-isomorphic to a $C^*$-algebra, algebra, that is $8\epsilon$-close, in Kadison-Kastler metric, to the same algebra isomorphic to a $C^*$-algebra. Hence for every $\epsilon>0$ there are $2^{\aleph_1}$ algebras that are amenable, nonseparable, close to each other and each of those is close to a an algebra isomorphic to a $C^*$-algebra. If we fix the set of points $\{x_\alpha\}$ in order to have $0$ as a complete accumulation point of $\{x_\alpha\}$, we can say that each of those $2^{\aleph_1}$ many algebras is $\epsilon$-close to a $C^*$-algebra but not isomorphic to one itself.

Lastly, we want to emphasize how the result contained in \cite{Vasilescu} relates to this problem: as was noted in \cite{CFO13}, the existence of a separable amenable subalgebra of $\mathcal B(H)$ that is not isomorphic to a $C^*$-algebra is equivalent to the existence of such an object inside $\prod M_n$, that is a finite von Neumann algebra. In particular, using the fact that any bounded group of invertible elements inside $\prod M_n$ is similar to a group of unitaries, we have the following
\begin{Cor}
Let $A\subseteq\prod M_n$ be a Banach algebra. Then the following conditions are equivalent:
\begin{itemize}
\item $A$ is isomorphic to a $C^*$-algebra;
\item $A$ is similar to a $C^*$-algebra;
\item There is a uniformly bounded group $G\subseteq GL(A)$ such that $A=\overline\spann G$.
\end{itemize}
\end{Cor}
Hence the existence of an amenable separable operator algebra not isomorphic to a $C^*$-algebra is equivalent to the existence of a separable $A\subseteq\prod M_n$ such that $A$ cannot be generated (as a Banach space) by a uniformly bounded group of invertible elements.

\bibliographystyle{amsplain}
\bibliography{NotUnitarizableRepresentation}
\end{document}